\documentclass[11pt]{article}
\usepackage{amssymb,amsmath,amsthm,amscd,latexsym}
\usepackage{mathrsfs}
\usepackage{mathrsfs}
\usepackage{amsfonts}
\usepackage{amsmath}
\usepackage{amssymb}
\usepackage{amscd}
 \usepackage{color}
 \usepackage{bm}
\usepackage{extpfeil}
\usepackage{booktabs}
\usepackage{pifont}

\renewcommand{\paragraph}{\roman{paragraph}}
\input cyracc.def

\newtheorem{thm}{\bfseries  Theorem}[section]
\newtheorem{lem}[thm]{\bfseries   Lemma}
\newtheorem{coro}[thm]{\bfseries   Corollary}

\newtheorem{con}[thm]{\bfseries   Conjecture}

\begin{document}
\title{\bf A new method for solving the equation $x^d+(x+1)^d=b$ in $\mathbb{F}_{q^4}$ where $d=q^3+q^2+q-1$
}
\author{Liqin Qian\thanks{Liqin Qian, School of Mathematical Sciences, Anhui University, Hefei, Anhui, 230601, China, {\tt qianliqin\_1108@163.com}}, Minjia Shi\thanks{Minjia Shi, School of Mathematical Sciences, Anhui University, Hefei, Anhui, 230601, China, {\tt smjwcl.good@163.com}}, Wei Lu\thanks{Wei Lu, School of Mathematics, Southeast University, Nanjing, Jiangsu, 210096, China, {\tt  luwei1010@139.com}}
}

\date{}
\maketitle
\begin{abstract} In this paper, we give a new method answer to a recent conjecture proposed by Budaghyan, Calderini, Carlet, Davidova and Kaleyski about the equation $x^d+(x+1)^d=b$ in $\mathbb{F}_{q^4}$, where $n$ is a positive integer, $q=2^n$ and $d=q^3+q^2+q-1$. In particular, we directly determine the differential spectrum of this power function $x^d$ using methods different from those in the literature. Compared with the methods in the literature, our method is more direct and simple.
\end{abstract}
{\bf Keywords:} Finite field, equation, power function, differential spectrum  \\
{\bf MSC(2010):} 11D04, 12E05, 12E12

\section{Introduction}
Let $n,m$ be two positive integers and $q=2^n$. Let $\mathbb{F}_{q}$ denote the finite field with $q$ elements.
An S-box is a vectorial Boolean function from $\mathbb{F}_{2^n}$ to $\mathbb{F}_{2^m}$, also called an $(n, m)$-function. The security of most modern block ciphers substantially relies on the cryptographic properties of their S-boxes, which are usually the only nonlinear components of these cryptosystems. In order to resist various kinds of cryptanalytic attacks, it is necessary to employ S-boxes with good cryptographic properties. One of the most powerful attacks against block ciphers is differential cryptanalysis.

In fact, for resistance of cryptosystem against differential attacks, S-boxes should have low differential uniformity.  The differential uniformity of a function $F$ was introduced by Nyberg \cite{N}, which as a measurement of the contribution of the function $F$ to the resistance of the block cipher against differential cryptanalysis. For any (finite) set $E$, $|E|$ denotes its cardinality. The multiplicative group of $\mathbb{F}_{2^n}$ will be denoted by $\mathbb{F}_{2^n}^*$. The \emph{differential uniformity} $\delta_F$ of an $(n, n)$-function $F$ is defined as
$$\delta_F=\max\{\delta_F(a,b): a\in \mathbb{F}_{2^n}^*,b\in \mathbb{F}_{2^n}\},$$
where $\delta_F(a,b)=|\{x\in \mathbb{F}_{2^n}: F(a+x)+F(x)=b\}|.$
 If $\delta_F=2$, then $F(x)$ is called an \emph{almost perfect nonlinear} (APN for short) function on $\mathbb{F}_{2^n}$. Since $a+x$ is a solution to $F(a+x)+F(x)=b$ whenever $x$ is, the differential uniformity $\delta_F$ must be even. Thus APN functions have the lowest possible differential uniformity and provide the best possible resistance to differential cryptanalysis.

In order to gain insight into the resistance of a cipher to differential attacks \cite{BS}, it is of interest to compute the differential spectra of power functions with low differential uniformity \cite{BCC,XY}. Let $F(x)=x^d$ be a power function over $\mathbb{F}_{2^n}$. Since $\delta_F(a,b)=\delta_F(1,\frac{b}{a^d})$ for $a\neq 0$. Let $w_i$ be defined as follows:
$$w_i=|\{b\in \mathbb{F}_{2^n}: \delta_F(1,b)=i\}|.$$ The \emph{differential spectrum} of $F(x)$ is the set $\mathbb{S}$ of $w_i$ ($i$ is even and $0\leq i \leq \delta_F$): 
$$\mathbb{S}=\{w_0,w_2,\ldots,w_{\delta_F}\}.$$
As a special class of functions over finite fields, power functions, have been extensively studied in the last decades due to their simple algebraic form and lower implementation cost in a hardware environment. Recently, Budaghyan \emph{et al.} \cite{BCCDK} presented some observations and computational data on the differential spectra of the power functions
\begin{equation}\label{eq17}
 F(x)=x^d~{\rm with}~d=\sum\limits_{i=1}^{k-1}2^{in}-1
\end{equation}
over $\mathbb{F}_{2^{nk}}$, where $n$ and $k$ are two positive integers. Note that this family of power functions includes some famous functions as special cases. For instance, when $n=1$, $F(x)=x^{2^{k}-3}$ coincides with the well-known inverse function; when $k=3$, $F(x)$ is cyclotomically equivalent to a Kasami type power function \cite{D1,HX,JW,K}; when $k=5$, $F(x)$ is the well-known Dobbertin function \cite{D}. For $k=2$, the differential properties have been investigated \cite{B,BCC1}, and the differential spectrum of the case $k=4$ was determined by Tu \cite{TWZTJ} and Kim \emph{et al.} \cite{KM}, which motivated
by the following conjecture proposed by Budaghyan \emph{et al.} based on the computational data.
\begin{con}\label{con1} \cite[Conjecture 27]{BCCDK} Let $d=q^3+q^2+q-1$ ($q=2^n$) and consider the power function $x^d$ over $\mathbb{F}_{q^4}$. Then the equation $x^d+(x+1)^d=b$ has $q^2$ solutions for one value of $b$; it has $q^2-q$ solutions for $q$ values of $b$; and has at most $2$ solutions for all  remaining points $b$.
\end{con}
\hspace*{-0.5cm}The recent work of Tu \emph{et al.} \cite{TWZTJ} gave the complete differential spectrum of $x^d~(d=q^3+q^2+q-1)$ and also presented an affirmative answer to \cite[Conjecture 27]{BCCDK}. Later, Kim and Mesnager \cite{KM} proposed a powerful approach to solve Conjecture \ref{con1} and and explicitly determined the set of $b$'s for which the equation $x^d+(x+1)^d=b$ has $i$ solutions for any positive integer $i$ simultaneously.

The purpose of this paper is to give a new method to solve Conjecture \ref{con1} and completely determine the differential spectrum of this power function $x^d~(d=q^3+q^2+q-1)$ using methods differ from \cite{KM, TWZTJ}.  The following
theorem is the main result of the paper.
\begin{thm}\label{thm} Let $b\in \mathbb{F}_{q^4}$ and $d=q^3+q^2+q-1$.
The number of solutions to the equation $x^d+(x+1)^d=b$ in $\mathbb{F}_{q^4}$ is equal to
\begin{equation*}
 \begin{aligned}
 \left\{
        \begin{array}{ll}
        q^2,~~~~~~b=1 ;\\
        q^2-q,~b\in \mu_{q+1}\backslash \{1\};\\
        2,~~~~~~~b\in \Lambda;\\
        0,~~~~~~~{\rm otherwise},\\
        \end{array}
  \right.
  \end{aligned}
\end{equation*}where $\Lambda$ is defined in Lemma \ref{lem7}.  Furthermore, the differential spectrum of $x^d$ is $$\mathbb{S}=\{w_0=(\frac{1}{2}q^3-1)(q+1),w_2=\frac{1}{2}q^3(q-1),w_{q^2}=1,w_{q^2-q}=q\}.$$
\end{thm}
\hspace*{-0.5cm}Compared with the methods in \cite{KM, TWZTJ}, the idea of our proof method is more direct and interesting. In particular, our method can directly determine the number of the set of $b$'s for which the
equation has 2 solutions.

This paper is organized as follows. In Section 2, we introduce notation and the necessary preliminaries required for the subsequent sections. In Section 3, we devote to the proof of our main result. Section 4 concludes the paper and elaborates our further research work.


\section{Preliminaries}
Let $r$ be a prime power and $m$ a positive integer. Define the trace $\mathbf{Tr}_{r}^{r^m}$ and norm $\mathbf{N}_r^{r^m}$ mappings from $\mathbb{F}_{r^m}$ to $\mathbb{F}_{r}$ by
$$ \mathbf{Tr}_{r}^{r^m}(x)=\sum_{i=0}^{m-1}x^{r^{i}},~\mathbf{N}_{r}^{r^m}(x)=\prod_{i=0}^{m-1}x^{r^{i}},$$ respectively.

Let $\mathbb{F}_{q^4}^*$ denote the multiplicative group of $\mathbb{F}_{q^4}$. For a positive integer $s$ with $s|(q^4-1)$, we define the set
$$\mu_s=\{x\in \mathbb{F}_{q^4}^*:x^s=1\},$$ which consists of all $s$-th roots of unity in $\mathbb{F}_{q^4}$. Note that $x^q=x$ if $x\in \mu_{q-1}(=\mathbb{F}_q^*)$, $x^q=x^{-1}$ if $x\in \mu_{q+1}$ and $x^{q^2}=x^{-1}$ if $x\in \mu_{q^2+1}$.

Next, we present below a decomposition of $\mathbb{F}_{q^4}^*$ that comes from \cite[Lemma 4]{KM}.
\begin{lem}\label{lem1} (\cite{KM}, Lemma 4) Let $n$ be a positive integer and $q=2^n$. $\mathbb{F}_{q^4}^*$ can be composed as $\mathbb{F}_{q^4}^*=\mu_{q-1}\cdot \mu_{q+1}\cdot \mu_{q^2+1}$. More precisely, for any $x\in \mathbb{F}_{q^4}^*$, there are unique elements $x_1\in\mu_{q-1}$, $x_2\in\mu_{q+1}$, $x_3\in\mu_{q^2+1}$ such that $x$ can be uniquely decomposed into the form
$$x=x_1x_2x_3.$$
\end{lem}
In fact, $x_1,x_2,x_3$ in Lemma \ref{lem1} have specific expressions.
\begin{lem} \label{lem2} Let $n$ be a positive integer and $q=2^n$. For any $x\in \mathbb{F}_{q^4}^*$, we let $n_1=\frac{1+q^2}{2}\cdot\frac{1+q}{2},n_2=\frac{1+q^2}{2}\cdot\frac{1-q}{2},n_3=\frac{1-q^2}{2}$ and $x_1=x^{n_1},x_2=x^{n_2},x_3=x^{n_3}$. Then $x=x_1x_2x_3,$ where $x_1\in\mu_{q-1}$, $x_2\in\mu_{q+1}$, $x_3\in\mu_{q^2+1}$.
\end{lem}
\begin{proof}
From Lemma \ref{lem1}, we have $\mathbb{F}_{q^4}^*=\mu_{q^2-1}\cdot \mu_{q+1}\cdot \mu_{q^2+1}$. Then there are elements $x_1\in\mu_{q-1}$, $x_2\in\mu_{q+1}$, $x_3\in\mu_{q^2+1}$ such that
\begin{equation}\label{eq1}x=x_1x_2x_3.
\end{equation}
By taking the $q^2$ power on both sides of Eq. (\ref{eq1}), we have \begin{equation}\label{eq2}x^{q^2}=x_1x_2x_3^{-1}.
\end{equation}And then we divide both sides of Eq. (\ref{eq1}) and Eq. (\ref{eq2}), we obtain $x^{1-q^2}=x_3^2$, which gives us $x_3=x^{\frac{1-q^2}{2}}$ since $q=2^n$.

Similarly, we also obtain  $x_1=x^{\frac{1+q^2}{2}\cdot\frac{1+q}{2}}$, $x_2=x^{\frac{1+q^2}{2}\cdot\frac{1-q}{2}}$.
\end{proof}

\begin{lem}\label{lem3} \cite[Corollary 3.79]{LN} Let $p=2$ and $q$ be a power of $2$. For any $\delta\in \mathbb{F}_{q}$, the number of solutions to the equation $x^2+x+\delta=0$ in $\mathbb{F}_q$ is equal to
\begin{equation*}
 \begin{aligned}
 \left\{
        \begin{array}{ll}
        0, & if~\mathbf{Tr}_2^{q}(\delta)=1;\\
        2, & if~\mathbf{Tr}_2^{q}(\delta)=0.\\
        \end{array}
  \right.
  \end{aligned}
\end{equation*}
\end{lem}
According to Lemma \ref{lem3} and \cite{DFHR,LCXM}, we can easily obtain the following result.
\begin{lem} \label{lem4}
Let $n$ be a positive integer and $q=2^n$. For any $a\in \mathbb{F}_{q^2}^*$, we have
\begin{itemize}
  \item [(1)] the equation $x+x^{-1}=a$
has two different solutions in $\mu_{q^2+1}\backslash\{1\}$ if and only $\mathbf{Tr}_2^{q^2}(\frac{1}{a})=1$.
  \item [(2)] the equation $x+x^{-1}=a$
has two different solutions in $\mathbb{F}_{q^2}\backslash\{1\}$ if and only $\mathbf{Tr}_2^{q^2}(\frac{1}{a})=0$.
\end{itemize}

\end{lem}




From Lemma \ref{lem4}, we have the following corollary.
\begin{coro}\label{coro1} Let $\Phi: \mathbb{F}_{q^4}^*\longrightarrow \mathbb{F}_{q^4}$ satisfy $\Phi(x)=x+x^{-1}=a$. Then
\begin{enumerate}
  \item [(1)] $\Phi(\mu_{q^2+1}\backslash \{1\})=\{a\in \mathbb{F}_{q^2}^*: \mathbf{Tr}_2^{q^2}\left(\frac{1}{a}\right)=1\};$
  \item [(2)] $\Phi(\mathbb{F}_{q^2}^*\backslash \{1\})=\{a\in \mathbb{F}_{q}^*: \mathbf{Tr}_2^{q^2}\left(\frac{1}{a}\right)=0\};$
  \item [(3)] $\Phi(\mu_{q+1}\backslash \{1\})=\{a\in \mathbb{F}_{q}^*: \mathbf{Tr}_2^{q}\left(\frac{1}{a}\right)=1\};$
  \item [(4)] $\Phi(\mathbb{F}_{q}^*\backslash \{1\})=\{a\in \mathbb{F}_{q}^*: \mathbf{Tr}_2^{q}\left(\frac{1}{a}\right)=0\}.$
\end{enumerate}
\end{coro}
According to the knowledge of systems of linear equations, the following result is easily obtained.
\begin{lem}\label{lem6}
Let $\beta\in \mathbb{F}_{q}^*$ and $\gamma \in \mathbb{F}_{q}$. Then the number of solutions to  the system of equations
\begin{equation*}
 \begin{aligned}
 \left\{
        \begin{array}{ll}
        \mathbf{Tr}_{2}^{q}(\beta x)=1;\\
        \mathbf{Tr}_{2}^{q}(\gamma x)=1,\\
        \end{array}
  \right.
  \end{aligned}
\end{equation*} in $\mathbb{F}_{q}$ is equal to
\begin{equation*}
 \begin{aligned}
 \left\{
        \begin{array}{ll}
        0,~if~\gamma =0;\\
        \frac{q}{2},~if~\gamma =\beta;\\
        \frac{q}{4},~if~\gamma \neq 0,\beta.\\
        \end{array}
  \right.
  \end{aligned}
\end{equation*}
\end{lem}

\begin{lem}\label{lem7}
Let $c\in \mathbb{F}_{q^4}^*$. Let $\mathfrak{t}=\mathbf{Tr}_{q^2}^{q^4}(c)$ and $\mathfrak{n}=\mathbf{N}_{q^2}^{q^4}(c)$. The set $$\Lambda =\left\{c\in \mathbb{F}_{q^4}^*: \mathfrak{t}'\neq 0~{\rm and}~\mathbf{Tr}_{2}^{q}\left(\frac{\mathfrak{n}'^{\frac{q+1}{2}}}{\mathfrak{t}'}\right)=1, {\rm where}~\mathfrak{t}'=\mathfrak{t}+\mathfrak{t}^q, \mathfrak{n}'=\mathfrak{n}+1\right\}.$$ Then the  cardinality of the set $\Lambda$ is equal to $\frac{1}{2}q^3(q-1)$.
\end{lem}
\begin{proof}
It is easy to know that $\mathfrak{t}'=\mathfrak{t}+\mathfrak{t}^q=c+c^{q^2}+(c+c^{q^2})^{q}=c+c^{q}+c^{q^2}+c^{q^3}=\mathbf{Tr}_q^{q^4}(c)$. Since $\mathfrak{t}'\neq 0$, we have $c\neq 0$. By Lemma \ref{lem2}, we obtain that there are unique elements $c_1\in\mu_{q-1}$, $c_2\in\mu_{q+1}$, $c_3\in\mu_{q^2+1}$ such that $c$ can be uniquely decomposed into the form $c=c_1c_2c_3$. Then we have $\mathfrak{n}'=c_1^2c_2^2+1, \mathfrak{t}'=c_1(c_2(c_3+c_3^{-1})+c_2^{-1}(c_3+c_3^{-1})^q)$. Since $\mathfrak{t}'\neq 0$, we have $c_3\neq 1$. It is easy to know that $\mathfrak{n}'^{\frac{1}{2}}=\mathfrak{n}'^{\frac{q^2}{2}}=(c_1^2c_2^2+1)^{\frac{q^2}{2}}=c_1c_2+1$ and then $\mathfrak{n}'^{\frac{q+1}{2}}=(c_1c_2+1)^{q+1}=c_1^2+c_1(c_2+c_2^{-1})+1.$ Thus, $\mathbf{Tr}_2^{q}\left(\frac{\mathfrak{n}'^{\frac{q+1}{2}}}{\mathfrak{t}'}\right)=\mathbf{Tr}_2^{q}\left(\frac{c_1+c_1^{-1}+c_2+c_2^{-1}}{c_2(c_3+c_3^{-1})+c_2^{-1}(c_3+c_3^{-1})^q}\right)=1$.
Set $a=c_3+c_3^{-1}\in \mathbb{F}_{q^2}^*,$ we have $\mathbf{Tr}_2^{q^2}\left(\frac{1}{a}\right)=1$ by Corollary \ref{coro1}(1) and $\mathbf{Tr}_2^{q}\left(\frac{c_1+c_1^{-1}+c_2+c_2^{-1}}{c_2a+c_2^{-1}a^q}\right)=1.$ From Lemma \ref{lem2}, there are unique elements $a_1\in \mu_{q-1},a_2\in \mu_{q+1}$ such that $a$ can be uniquely decomposed into the form $a=a_1a_2.$ So, we have
\begin{equation}\label{eq15}
 \begin{aligned}
 \left\{
        \begin{array}{ll}
        1=\mathbf{Tr}_2^{q^2}\left(\frac{1}{a}\right)=\mathbf{Tr}_2^{q}\left(\frac{1}{a_1}(a_2+a_2^{-1})\right);\\
        1=\mathbf{Tr}_2^{q}\left(\frac{c_1+c_1^{-1}+c_2+c_2^{-1}}{c_2a+c_2^{-1}a^q}\right)=\mathbf{Tr}_2^{q}\left(\frac{1}{a_1}\cdot\frac{c_1+c_1^{-1}+c_2+c_2^{-1}}{a_2c_2+a_2^{-1}c_2^{-1}}\right).\\
        \end{array}
  \right.
  \end{aligned}
\end{equation}
Set $x=\frac{1}{a_1},\beta=a_2+a_2^{-1},\gamma=\frac{c_1+c_1^{-1}+c_2+c_2^{-1}}{a_2c_2+a_2^{-1}c_2^{-1}}$, we have $\beta\neq 0$ and $a_2\in \mu_{q+1}\backslash \{1\}$. Then Eq. (\ref{eq15}) is  equivalent to
\begin{equation}\label{eq16}
 \begin{aligned}
 \left\{
        \begin{array}{ll}
        \mathbf{Tr}_{2}^{q}(\beta x)=1;\\
        \mathbf{Tr}_{2}^{q}(\gamma x)=1.\\
        \end{array}
  \right.
  \end{aligned}
\end{equation}
Note that $\gamma=\beta$ if and only if $c_1+c_1^{-1}=a_2^2c_2+(a_2^2c_2)^{-1}$, i.e., $c_1=1$ and $a_2^2c_2=1$ by Corollary \ref{coro1}(3-4). 
Note that $a_2c_2+a_2^{-1}c_2^{-1}\neq 0$, i.e., $c_2\neq a_2^{-1}$.

If $c_1\neq1$, then $\gamma\neq \beta$ and $\gamma \neq 0.$ From Lemma \ref{lem6}, we obtain that Eq. (\ref{eq16}) has $\frac{q}{4}$ solutions in $\mathbb{F}_q$. Therefore, there are $(q-2)\cdot q\cdot q\cdot \frac{q}{4}\cdot 2=\frac{1}{2}q^4-q^3$ choices for $(c_1,a_2,c_2,a_1,c_3)$.

If $c_1=1$ and $c_2=1$, then $\gamma =0$. From Lemma \ref{lem6}, we obtain that Eq. (\ref{eq16}) has no solution in $\mathbb{F}_q$. Therefore, there are $1\cdot q\cdot 1\cdot 0\cdot 0=0$ choices for $(c_1,a_2,c_2,a_1,c_3)$.

If $c_1=1$ and $c_2=a_2^{-2}$, then $\gamma=\beta$. From Lemma \ref{lem6}, we obtain that Eq. (\ref{eq16}) has $\frac{q}{2}$ solutions in $\mathbb{F}_q$. Therefore, there are $1\cdot q\cdot 1\cdot \frac{q}{2}\cdot 2=q^2$ choices for $(c_1,a_2,c_2,a_1,c_3)$.

If $c_1=1$ and $c_2\in \mu_{q+1}\backslash \{1,a_2^{-2}\}$, then $\gamma\neq\beta$ and $\gamma \neq 0.$ From Lemma \ref{lem6}, we obtain that Eq. (\ref{eq16}) has $\frac{q}{4}$ solutions in $\mathbb{F}_q$. Therefore, there are $1\cdot q\cdot (q-2)\cdot \frac{q}{4}\cdot 2=\frac{1}{2}q^3-q^2$ choices for $(c_1,a_2,c_2,a_1,c_3)$.

Hence, the number of the set $\Lambda$ is equal to $(\frac{1}{2}q^4-q^3)+0+q^2+(\frac{1}{2}q^3-q^2)=\frac{1}{2}q^3(q-1)$.
\end{proof}

\section{Proof of Theorem \ref{thm}}
Let $q=2^n$ and $n$ be a positive integer. Let $b\in \mathbb{F}_{q^4}$. Let $d=q^3+q^2+q-1$. We prove the following.

To complete the proof of Theorem \ref{thm}, it is sufficient to determine the number of solutions in $\mathbb{F}_{q^4}$ 
 of the equation
\begin{equation}\label{eq4}
 x^d+(x+1)^d=b
\end{equation}
when $b$ runs through $\mathbb{F}_{q^4}$.

Assume that $y=x+1$, then we obtain that Eq. (\ref{eq4}) is equivalent to
\begin{equation}\label{eq5}
 \begin{aligned}
 \left\{
        \begin{array}{ll}
        x+y=1;\\
        x^d+y^d=b,\\
        \end{array}
  \right.
  \end{aligned}x,y\in \mathbb{F}_{q^4}.
\end{equation}
We assume that $\Omega:=\mathbb{F}_{q^4}\times \mathbb{F}_{q^4}$ and $\Omega^*:=\mathbb{F}_{q^4}^*\times \mathbb{F}_{q^4}^*$. For any subset $\Theta$ of $\Omega$, we define the following parameter
$$N_b(\Theta):=|\{(x,y)\in \Theta: (x,y)~{\rm is~the~solution~to~Eq.~(\ref{eq5})}\}|.$$
It is obvious that
\begin{equation*}
 N_b(\{(1,0),(0,1)\})=\begin{aligned}
 \left\{
        \begin{array}{ll}
        2,~b=1 ;\\
        0,~b\neq 1,\\
        \end{array}
  \right.
  \end{aligned}
\end{equation*}
and $N_b(\Omega)=N_b(\{(1,0),(0,1)\})+N_b(\Omega^*)$.

Thus, we just need to determine $$N_b(\Omega^*):=|\{(x,y)\in \Omega^*: (x,y)~{\rm is~the~solution~to~Eq.~(\ref{eq5})}\}|.$$

Due to $x, y\in \mathbb{F}_{q^4}^*$, and by Lemma \ref{lem2}, we can assume that $x=x_1x_2x_3$ and $y=y_1y_2y_3$, where $x_1,y_1\in \mu_{q-1}$, $x_2,y_2\in \mu_{q+1}$ and $x_3,y_3\in \mu_{q^2+1}$. Then, we have $$x^d=x^{q^3+q^2+q+1-2}=\frac{\mathbf{N}_q^{q^4}(x)}{x^2}=\frac{\mathbf{N}_q^{q^4}(x_1x_2x_3)}{x_1^2x_2^2x_3^2}=x_1^{2}x_2^{-2}x_3^{-2}.$$
Similarly, $y^d=y_1^{2}y_2^{-2}y_3^{-2}$. Hence, the Eq. (\ref{eq5}) is equivalent to
\begin{equation}\label{eq6}
 \begin{aligned}
 \left\{
        \begin{array}{ll}
        x_1x_2x_3+y_1y_2y_3=1;\\
        x_1^{2}x_2^{-2}x_3^{-2}+y_1^{2}y_2^{-2}y_3^{-2}=b.\\
        \end{array}
  \right.
  \end{aligned}
\end{equation}
Next, we consider the equation $x_1^{2}x_2^{-2}x_3^{-2}+y_1^{2}y_2^{-2}y_3^{-2}=b$, by taking the $\frac{q^2}{2}$ power on both sides of it, we have
$x_1x_2^{-1}x_3+y_1y_2^{-1}y_3=b'$, where $b'=b^{\frac{q^2}{2}}$. Therefore, the Eq. (\ref{eq6}) is equivalent to
\begin{equation}\label{eq7}
 \begin{aligned}
 \left\{
        \begin{array}{ll}
        x_2\cdot x_1x_3+y_2\cdot y_1y_3=1;\\
        x_2^{-1}\cdot x_1x_3+y_2^{-1}\cdot y_1y_3=b'.\\
        \end{array}
  \right.
  \end{aligned}
\end{equation}
It obvious that the determinant of the coefficients of the system of Eqs. (\ref{eq7}) is
\begin{eqnarray*}
\Delta=\begin{array}{|cc|}
   x_2 &    y_2      \\
    x_2^{-1}  &  y_2^{-1}   \\
\end{array}=\frac{(x_2+y_2)^2}{x_2y_2}.
\end{eqnarray*}

Next, let's discuss the rest of the proof in two cases  depending
on $\Delta=0$ or $\Delta\neq 0$.

\textbf{(1) Case 1: $\Delta=0$, i.e., $x_2=y_2$}

In this case, we obtain that the Eq. (\ref{eq7}) is equivalent to
\begin{equation}\label{eq8}
 \begin{aligned}
 \left\{
        \begin{array}{ll}
        x_1x_3+y_1y_3=x_2^{-1};\\
        x_1x_3+y_1y_3=b'x_2.\\
        \end{array}
  \right.
  \end{aligned}
\end{equation}

$\bullet$ If $b'\notin \mu_{q+1}$, then Eq. (\ref{eq8}) has no solution in $\mathbb{F}_{q^4}$.

$\bullet$ If $b'\in \mu_{q+1}$, then we have $x_2=y_2=b'^{-\frac{1}{2}}\in \mu_{q+1}$, and $x_1x_3+y_1y_3=b''$, where $b''=b'^{\frac{1}{2}}$. By taking the $q^2$ power on both sides of the equation $x_1x_3+y_1y_3=b''$, we have $x_1x_3^{-1}+y_1y_3^{-1}=b''$. Hence, the Eq. (\ref{eq8}) is equivalent to
\begin{equation}\label{eq9}
 \begin{aligned}
 \left\{
        \begin{array}{ll}
        y_1y_3=x_1x_3+b'';\\
        y_1y_3^{-1}=x_1x_3^{-1}+b''.\\
        \end{array}
  \right.
  \end{aligned}
\end{equation}
Let's multiply both sides of these two equations in (\ref{eq9}), we have
\begin{equation}\label{eq10}
x_3+x_3^{-1}=\frac{y_1^2+x_1^2+b''^2}{x_1b''}\in \mathbb{F}_{q^2}.
\end{equation}

If $\frac{y_1^2+x_1^2+b''^2}{x_1b''}=0$, then
$y_1+x_1=b''=1$, and $x_3=y_3=1$. Note that $|\{(x_1,y_1)\in \mathbb{F}_{q}^*\times \mathbb{F}_{q}^*: x_1+y_1=1\}|=q-2$. Hence, Eq. (\ref{eq8}) has $q-2$ solutions  in $\mathbb{F}_{q^4}$.

If $\frac{y_1^2+x_1^2+b''^2}{x_1b''}\neq0$, then we have the following by Lemma \ref{lem4}: if $\mathbf{Tr}_2^{q^2}\left(\frac{x_1b''}{y_1^2+x_1^2+b''^2}\right)=0,$ then Eq. (\ref{eq10}) has no solution in $\mu_{q^2+1}\backslash\{1\}$; if $\mathbf{Tr}_2^{q^2}\left(\frac{x_1b''}{y_1^2+x_1^2+b''^2}\right)=1,$ then Eq. (\ref{eq10}) has two different solutions in $\mu_{q^2+1}\backslash\{1\}$.


In the following, we discuss how many pairs of $(x_1,y_1)\in \mathbb{F}_q^*\times \mathbb{F}_q^*$ such that $\mathbf{Tr}_2^{q^2}\left(\frac{x_1b''}{y_1^2+x_1^2+b''^2}\right)$ $=1$. Set $z_1=x_1+y_1$. Then, we have
$1=\mathbf{Tr}_2^{q^2}\left(\frac{x_1b''}{z_1^2+b''^2}\right)=\mathbf{Tr}_2^{q}\left(\mathbf{Tr}_q^{q^2}\left(\frac{x_1b''}{z_1^2+b''^2}\right)\right)=
\mathbf{Tr}_2^{q}\left(x_1A(z_1)\right)$, where $A(z_1)=\frac{(b''+b''^{-1})(z_1^2+1)}{(z_1^2+b''^{2})(z_1^2+b''^{-2})}$. Note that $z_1\neq 1$ and $b''\neq 1$.
Let $N=|\{(x_1,z_1)\in \mathbb{F}_q^*\times \mathbb{F}_{q}: x_1\neq z_1, z_1\neq 1, \mathbf{Tr}_2^{q}\left(x_1A(z_1)\right)=1\}|$. Then
\begin{eqnarray*}
  N &=& |\{(x_1,0): x_1\in \mathbb{F}_q^*, \mathbf{Tr}_2^{q}\left(x_1A(0)\right)=1\}\cup \\&& \{(x_1,z_1)\in \mathbb{F}_q^*\times \mathbb{F}_{q}^*: z_1\neq 1, x_1\neq z_1, \mathbf{Tr}_2^{q}\left(x_1A(z_1)\right)=1\}| \\
  &=&|\{(x_1,0): x_1\in \mathbb{F}_q^*, \mathbf{Tr}_2^{q}\left(x_1A(0)\right)=1\}|+\\&& |\{(x_1,z_1)\in \mathbb{F}_q^*\times \mathbb{F}_{q}^*: z_1\neq 1, x_1\neq z_1, \mathbf{Tr}_2^{q}\left(x_1A(z_1)\right)=1\}|\\
   &=& \frac{q}{2}+(q-2)\cdot \frac{q}{2}=\frac{q}{2}(q-1).
\end{eqnarray*} 

Hence, for $b'\in \mu_{q+1}\backslash \{1\}$, Eq. (\ref{eq8}) has $2N=2(\frac{q}{2}(q-1))=q^2-q$ solutions in $\mathbb{F}_{q^4}$.

\textbf{(2) Case 2: $\Delta\neq 0$, i.e., $x_2\neq y_2$}

According to Cramer's Rule, and combining with Eq. (\ref{eq7}), we have

\begin{equation}\label{eq11}
 \begin{aligned}
 \left\{
        \begin{array}{ll}
        x_1x_3=\frac{\begin{array}{|cc|}
   1 &    y_2     \\
    b'  &  y_2^{-1}   \\
\end{array}}{\begin{array}{|cc|}
   x_2 &    y_2      \\
    x_2^{-1}  &  y_2^{-1}   \\
\end{array}}=\frac{y_2b'+y_2^{-1}}{x_2y_2^{-1}+y_2x_2^{-1}}=:S;\\
        y_1y_3=\frac{\begin{array}{|cc|}
   x_2 &    1      \\
    x_2^{-1}  &  b'   \\
\end{array}}{\begin{array}{|cc|}
   x_2 &    y_2      \\
    x_2^{-1}  &  y_2^{-1}   \\
\end{array}}=\frac{x_2b'+x_2^{-1}}{x_2y_2^{-1}+y_2x_2^{-1}}=:T.\\
        \end{array}
  \right.
  \end{aligned}
\end{equation}

When $S=0$ or $T=0$, we obtain that Eq. (\ref{eq11}) has no solution in $\mu_{q-1}$.

When $S\neq0$ and $T\neq0$, by Lemma \ref{lem2}, we obtain $x_1x_3=x_1\cdot1\cdot x_3=S^{n_1}\cdot S^{n_2} \cdot S^{n_3}$ and $y_1y_3=y_1\cdot 1\cdot y_3=T^{n_1}\cdot T^{n_2} \cdot T^{n_3}$, where $n_1=\frac{1+q^2}{2}\cdot \frac{1+q}{2}, n_2=\frac{1+q^2}{2}\cdot \frac{1-q}{2}$ and $n_3=\frac{1-q^2}{2}$. Hence, we have the following
\begin{equation*}
 \begin{aligned}
 \left\{
        \begin{array}{ll}
x_1=S^{n_1};\\
        1=S^{n_2};\\
        x_3=S^{n_3};
        \end{array}
  \right.
  \end{aligned}{\rm and}~\begin{aligned}
 \left\{
        \begin{array}{ll}
y_1=T^{n_1};\\
        1=T^{n_2};\\
        y_3=T^{n_3}.
        \end{array}
  \right.
  \end{aligned}
\end{equation*}
Note that $x_1,y_1, x_3, y_3$ are determined by $x_2, y_2$. Hence, we just have to solve the system of equations with respect to $x_2$ and $y_2$:
\begin{equation}\label{eq12}
 \begin{aligned}
 \left\{
        \begin{array}{ll}
        S^{n_2}=1;\\
        T^{n_2}=1.\\
        \end{array}
  \right.
  \end{aligned}
\end{equation}
Due to $(x_2y_2^{-1}+y_2x_2^{-1})^q=y_2x_2^{-1}+x_2y_2^{-1}$, we have $x_2y_2^{-1}+y_2x_2^{-1}\in \mathbb{F}_q^*$. Hence,
\begin{eqnarray*}
  && S^{n_2}=\left(\frac{y_2b'+y_2^{-1}}{x_2y_2^{-1}+y_2x_2^{-1}}\right)^{\frac{1+q^2}{2}\cdot \frac{1-q}{2}}=1 \\
 &\Leftrightarrow& (y_2b'+y_2^{-1})^{(1+q^2)\cdot (1-q)}=(x_2y_2^{-1}+y_2x_2^{-1})^{(1+q^2)\cdot (1-q)}=1\\
 &\Leftrightarrow& (y_2^{-2}+b'^{q^2+1}y_2^2+b'+b'^{q^2})^{1-q}=1.
\end{eqnarray*}
Set $\mathfrak{n}=\mathbf{N}_{q^2}^{q^4}(b')=b'^{q^2+1}\in \mathbb{F}_{q^2}, \mathfrak{t}=\mathbf{Tr}_{q^2}^{q^4}(b')=b'+b'^{q^2}\in \mathbb{F}_{q^2}.$ Then we have
\begin{eqnarray*}
  && y_2^{-2}+\mathfrak{n}y_2^2+\mathfrak{t}=(y_2^{-2}+\mathfrak{n}y_2^2+\mathfrak{t})^q=y_2^{2}+\mathfrak{n}^qy_2^{-2}+\mathfrak{t}^q \\
 &\Leftrightarrow& (\mathfrak{n}+1)y_2^2+(\mathfrak{t}+\mathfrak{t}^q)+(\mathfrak{n}^q+1)y_2^{-2}=0\\
 &\Leftrightarrow& (\mathfrak{n}+1)y_2^4+(\mathfrak{t}+\mathfrak{t}^q)y_2^{2}+(\mathfrak{n}^q+1)=0.
\end{eqnarray*}
Set $z=y_2^2\in \mu_{q+1}$. Then we have
\begin{equation}\label{eq13}
(\mathfrak{n}+1)z^2+(\mathfrak{t}+\mathfrak{t}^q)z+(\mathfrak{n}^q+1)=0.
\end{equation}
Similarly, from $T^{n_2}=1$, we can also get Eq. (\ref{eq13}). Hence, $x_2^2,y_2^2$ are two different solutions to Eq. (\ref{eq13}).
Next, let's discuss the rest of the proof in four cases.

\textcircled{1} When $\mathfrak{n}+1=0$ and $\mathfrak{t}+\mathfrak{t}^q=0$, i.e., $b'=1$, we obtain that Eq. (\ref{eq13}) has $q+1$ different solutions in $\mathbb{F}_{q^2}$. According to Lemma \ref{lem2}, we know that $x_1x_3$ has the decomposition in Eq. (\ref{eq11}) when $S\neq 0$, i.e., $\frac{y_2b'+y_2^{-1}}{x_2y_2^{-1}+y_2x_2^{-1}}\neq 0$. By $b'=1$ and $y_2b'+y_2^{-1}\neq 0$, we obtain $y_2\neq 1.$ Similarly, we also obtain $x_2\neq 1$. Hence, $x_2^2, y_2^2\in \mu_{q+1}\backslash\{1\}$ and $x_2^2\neq y_2^2$.
In this case, there are $q(q-1)=q^2-q$  choices for $x_2^2,y_2^2$.

\textcircled{2} When $\mathfrak{n}+1=0$ and $\mathfrak{t}+\mathfrak{t}^q\neq0$, we obtain $z=0$. In this case, there is no solution for $x_2^2,y_2^2$.

\textcircled{3} When $\mathfrak{n}+1\neq0$ and $\mathfrak{t}+\mathfrak{t}^q=0$, we have $(\mathfrak{n}+1)z^2+\mathfrak{n}^q+1=0$, which has only a unique solution for $z$. Hence, there is no solution for $x_2^2,y_2^2$.

\textcircled{4} When $\mathfrak{n}+1\neq0$ and $\mathfrak{t}+\mathfrak{t}^q\neq0$, we let $\mathfrak{n}'=\mathfrak{n}+1\in \mathbb{F}_{q^2}^*$ and $\mathfrak{t}'=\mathfrak{t}+\mathfrak{t}^q\in \mathbb{F}_{q}^*$. Hence, Eq. (\ref{eq13}) is equivalent to
\begin{eqnarray}\label{eq14}
  &&\mathfrak{n}'z^2+\mathfrak{t}'z+\mathfrak{n}'^q=0 \notag\\
 &\Leftrightarrow& z^2+\frac{\mathfrak{t}'}{\mathfrak{n}'}z+\mathfrak{n}'^{q-1}=0\notag\\
 &\Leftrightarrow& \left(\frac{\mathfrak{n}'}{\mathfrak{t}'}z\right)^2+\frac{\mathfrak{n}'}{\mathfrak{t}'}z+\frac{\mathfrak{n}'^{q+1}}{\mathfrak{t}'^2}=0.
\end{eqnarray}Since $\frac{\mathfrak{n}'^{q+1}}{\mathfrak{t}'^2}\in \mathbb{F}_{q}$, we have $\mathbf{Tr}_2^{q^2}\left(\frac{\mathfrak{n}'^{q+1}}{\mathfrak{t}'^2}\right)=\mathbf{Tr}_2^{q}\left(\mathbf{Tr}_q^{q^2}\left(\frac{\mathfrak{n}'^{q+1}}{\mathfrak{t}'^2}\right)\right)=\mathbf{Tr}_2^{q}(0)=0$. By Lemma \ref{lem3}, we obtain that Eq. (\ref{eq14})  has two different solutions in $\mathbb{F}_{q^2}$.

Without loss of generality, let's assume that $z_1,z_2\in \mathbb{F}_{q^2}$ are the two different solutions to the equation $\mathfrak{n}'z^2+\mathfrak{t}'z+\mathfrak{n}'^q=0$. According to Viete theorem, we have $z_1\cdot z_2=\mathfrak{n}'^{q-1}\in \mu_{q+1}$, which  implies that $z_1\neq z_2\in \mu_{q+1}$ if and only if $\frac{z_1}{z_2}\in \mu_{q+1}\backslash \{1\}.$ Hence, there are two choices for $x_2^2,y_2^2$.

It follows from $\mathfrak{t}'\neq 0$ that $b''\neq 0$, i.e., $b\neq 0$. Next, let's talk about how many $b\in \mathbb{F}_{q^4}^*$ such that $\frac{z_1}{z_2}\in \mu_{q+1}\backslash \{1\}$. In addition, we have $\frac{z_1}{z_2}+\frac{1}{\frac{z_1}{z_2}}=\frac{z_1^2+z_2^2}{z_1z_2}=\frac{(z_1+z_2)^2}{z_1z_2}=\frac{(\frac{\mathfrak{t}'}{\mathfrak{n}'})^2}{\frac{\mathfrak{n}'^q}{\mathfrak{n}'}}=\frac{\mathfrak{t}'^2}{\mathfrak{n}'^{q+1}}\in \mathbb{F}_q^*$. From Lemma \ref{lem4}, we obtain that $\frac{z_1}{z_2}\in \mu_{q+1}\backslash \{1\}$ if and only if $\mathbf{Tr}_2^{q}\left(\frac{\mathfrak{n}'^{q+1}}{\mathfrak{t}'^2}\right)=1$, i.e., $\mathbf{Tr}_2^{q}\left(\frac{\mathfrak{n}'^{\frac{q+1}{2}}}{\mathfrak{t}'}\right)=1$.
Set $$\Lambda =\left\{b\in \mathbb{F}_{q^4}^*: \mathfrak{t}'\neq 0~{\rm and}~\mathbf{Tr}_{2}^{q}\left(\frac{\mathfrak{n}'^{\frac{q+1}{2}}}{\mathfrak{t}'}\right)=1, {\rm where}~\mathfrak{t}'=\mathfrak{t}+\mathfrak{t}^q, \mathfrak{n}'=\mathfrak{n}+1\right\}.$$ It follows from Lemma \ref{eq7} that $|\Lambda|=\frac{1}{2}q^3(q-1)$.

To sum up, we obtain
\begin{equation*}
 N_b(\Omega^*)=\begin{aligned}
 \left\{
        \begin{array}{ll}
        q^2-2,~b=1 ;\\
        q^2-q,~b\in \mu_{q+1}\backslash \{1\};\\
        2,~~~~~~~b\in \Lambda ;\\
        0,~~~~~~~{\rm otherwise}.\\
        \end{array}
  \right.
  \end{aligned}
\end{equation*}
Hence, \begin{equation*}
 N_b(\Omega)=N_b(\{(1,0),(0,1)\})+N_b(\Omega^*)=\begin{aligned}
 \left\{
        \begin{array}{ll}
        q^2,~~~~~~b=1 ;\\
        q^2-q,~b\in \mu_{q+1}\backslash \{1\};\\
        2,~~~~~~~~b\in \Lambda ;\\
        0,~~~~~~~~{\rm otherwise}.\\
        \end{array}
  \right.
  \end{aligned}
\end{equation*}

\section{Concluding remarks}
In this paper, we presented a new method to solve the equation $x^d+(x+1)^d=b$ in $\mathbb{F}_{q^4}$, where $d=q^3+q^2+q-1$. Our method is interesting and promising, and is different from Kim et. al. \cite{KM} and Tu et. al. \cite{TWZTJ}. On the one hand, our method differs from the one adopted by Kim and Mesnager mainly in that we can not only give the cardinality  of solutions of the equation but also directly determine the number of the set of $b$'s for
which the equation has 2 solutions. On the other hand,  our method differs from the one used by Tu et. al. mainly in that our algebraic method could be helpful to
solve the equation $x^d+(x+1)^d=b$  over $\mathbb{F}_{2^{kn}}$ for other
values of $k$ where $d=\sum\limits_{i=1}^{k-1}2^{in}-1$, which we are working on. In addition, we are using our methods to determine the $c$-differential spectrum of the power function $x^d$ over $\mathbb{F}_{q^4}$ where $d=q^3+q^2+q-1$.

As an immediate consequence of our results, we proved Conjecture 27 and directly determined the differential spectrum of this power function $x^d$ using methods different from those in \cite{KM, TWZTJ}.


\end{document}